\documentclass{article}
\usepackage{cite}
\usepackage{amsmath,amssymb,amsthm,mathrsfs}
\usepackage{titlesec,hyperref}
\usepackage{color}
\usepackage{fancyhdr}
\usepackage[margin=2.5cm]{geometry}
\usepackage{enumerate}
\usepackage[integrals]{wasysym}
\usepackage{bm}
\usepackage{graphicx}
\usepackage{caption}
\usepackage{float}
\usepackage{subfigure}

\newtheorem{theo}{Theorem}[section]
\newtheorem{lemm}{Lemma}[section]

\newtheorem{prop}{Proposition}[section]
\newtheorem{rema}{Remark}[section]
\numberwithin{equation}{section}
\def\ud{{\rm d}}
\numberwithin{equation}{section}

\def\ud{{\rm d}}

\allowdisplaybreaks

\begin{document}

\title{A new result for the local well-posedness of the Camassa-Holm type equations in critial Besov spaces $B^{1+\frac{1}{p}}_{p,1},1\leq p<+\infty$}
\author{
Weikui $\mbox{Ye}^1$ \footnote{email: 904817751@qq.com} \quad and \quad
Zhaoyang $\mbox{Yin}^{2}$ \footnote{email: mcsyzy@mail.sysu.edu.cn}
\quad and\quad
Yingying $\mbox{Guo}^{3}$ \footnote{email: guoyy35@fosu.edu.cn}\\
$^1\mbox{Institute}$ of Applied Physics and Computational Mathematics,\\
P.O. Box 8009, Beijing 100088, P. R. China\\
$^2\mbox{Department}$ of Mathematics, Sun Yat-sen University,\\
Guangzhou, 510275, China\\
$^3\mbox{School}$ of Mathematics and Big Data, Foshan University,\\
Foshan, 528000, China
}
\date{}
\maketitle
\begin{abstract}
For the famous Camassa-Holm equation, the well-posedness in $B^{1+\frac{1}{p}}_{p,1}(\mathbb{R})$ with $1\leq p\leq2$ and the ill-posedness in $B^{1+\frac{1}{p}}_{p,r}(\mathbb{R})$ with $1\leq p\leq+\infty,\ 1<r\leq+\infty$ had been studied in \cite{d1,d2,glmy}. That is to say, it left an open problem in the critical case $B^{1+\frac{1}{p}}_{p,1}(\mathbb{R})$ with $2<p\leq+\infty$ proposed by Danchin in \cite{d1,d2}. In this paper, we solve this problem. The main difficulty is to prove the uniqueness, which usually needs to use the Moser-type inequality, resulting in the index $p$ belongs to $[1,2]$. To overcome the difficulty, inspired by Linares, Ponce and Thomas \cite{lps}, we combine the Lagrange coordinate transformation and small time conditions to avoid using the Moser-type inequality. As a result, we obtain the local well-posedness for the Camassa-Holm equation in critical Besov spaces $B^{1+\frac{1}{p}}_{p,1}(\mathbb{R})$ with $1\leq p<+\infty$. It is worth mentioning that our method is suitable for many Camassa-Holm type equations such as the Novikov equation and the two-component Camassa-Holm system, which can also improve their index on the local well-posedness.
\end{abstract}

\tableofcontents

\section{Introduction}
\par
In this paper, we consider the Cauchy problem for the Camassa-Holm type equations which have attracted much attention in recent twenty years.

The first one of the Camassa-Holm type equations is the following Camassa-Holm (CH) equation \cite{ch}
\begin{equation*}
u_t-u_{xxt}+3uu_{x}=2u_{x}u_{xx}+uu_{xxx}.
\end{equation*}
The CH is completely integrable\cite{c3,cgi,cmc},
has a bi-Hamiltonian structure \cite{c1,or} and other important features. One of the remarkable features is that it has the single peakon solutions
\begin{align*}
\varphi(t,x)=ce^{-|x-ct|},\quad  c\in\mathbb{R}	
\end{align*}
and the multi-peakon solutions \cite{achm,ch,cht}
\begin{align*}
u(t,x)=\sum\limits_{i=1}^{N}p_i(t)e^{-|x-q_i(t)|}
\end{align*}
where $p_i,\ q_i$ satisfy the Hamilton system
\begin{equation*}
\left\{\begin{array}{ll}
\frac{{\ud}p_i}{{\ud}t}=-\frac{\partial H}{\partial q_i}=\sum\limits_{i\neq j}p_ip_j{\rm sign}(q_i-q_j)e^{-|q_i-q_j|},\\
\frac{{\ud}q_i}{{\ud}t}=\frac{\partial H}{\partial p_i}=\sum\limits_{j}p_je^{-|q_i-q_j|}
\end{array}\right.
\end{equation*}
with the Hamiltonian $H=\frac{1}{2}\sum_{i,j=1}^{N}p_ip_je^{-|q_i-q_j|}$. It is shown that those peaked solitons were orbitally stable in the energy space \cite{c5,cs,em,t}. Another remarkable feature of the CH equation is the so-called wave breaking phenomena \cite{c2,ce3,ce6,lio}.
Moreover, the CH equation is locally well-posed and ill-posed, has global strong solutions, global weak solutions, global conservative solutions and dissipative solutions \cite{ce2,ce4,ce5,ce1,d1,d2,liy,glmy,cmo,bc1,bc2,bcz,hr1,hr3,lps}.

The second one is the following Novikov equation \cite{n}
\begin{align*}
u_{t}-u_{xxt}=3uu_{x}u_{xx}+u^{2}u_{xxx}-4u^{2}u_{x}.
\end{align*}
The Novikov equation has cubic nonlinear term which is different from the CH equation who has only quadratic nonlinear term. As well, the Novikov equation is integrable and has a bi-Hamiltonian structure \cite{hw}. The local well-posedness, global strong solutions, wave breaking solutions, global weak solutions of the Novikov equation in Sobolev spaces and Besov spaces were investigated in \cite{wy2,wy3,ylz2}.

The third one is the following two-component Camassa-Holm system \cite{or,ci}
\begin{equation*}
\left\{\begin{array}{ll}
u_{t}+uu_{x}=-\partial_{x}(1-\partial_{xx})^{-1}(u^{2}+\frac{1}{2}u^{2}_{x}+\frac{1}{2}\rho^{2}),\\
\rho_{t}+u\rho_{x}=-u_{x}\rho.
\end{array}\right.
\end{equation*}
The two-component Camassa-Holm system is integrable, and has a bi-Hamiltonian structure, global strong solutions, wave breaking phenomenon, global weak solutions and so on \cite{ci,gl1,gl2,gy1,gy2,ghr1}.

For the well-posedness of the above Camassa-Holm type equations, it was only proved in the spaces $B^{s}_{p,r}$ with $s>\max\{\frac{3}{2},1+\frac{1}{p}\}$ or $s=1+\frac{1}{p}$ with $p\in[1,2]$, $r=1$ ($1+\frac{1}{p}\geq\frac{3}{2}$) in \cite{d1,d2,liy}. Meanwhile, Guo et al. \cite{glmy} established
the ill-posedness for the Camassa-Holm type equations in Besov spaces $B^{1+\frac 1 p}_{p,r}$ with $p\in[1,+\infty],\ r\in(1,+\infty]$. This implies $B^{1+\frac 1 p}_{p,1}$ is the critical Besov space for the Camassa-Holm type equations. However,  whether the above Camassa-Holm type equations will be well-posed or not in $B^{1+\frac{1}{p}}_{p,1}$ with $p\in(2,+\infty]$ ($1+\frac{1}{p}<\frac{3}{2}$) is still an open problem. Therefore, in this paper, we aim to solve this problem in critial Besov spaces $B^{1+\frac{1}{p}}_{p,1}$ with $p\in[1,+\infty)$, which will imply the index $\frac{3}{2}$ is not necessary and can improve the results in many papers, such as \cite{d1,d2,liy}.

The main difficulty is to prove the uniqueness. For instance, one should use the following Moser-type inequality
\begin{align}\label{mosher}
\|fg\|_{{B}^{s_1+s_2-\frac{d}{p}}_{p,1}}\leq C\|f\|_{{B}^{s_1}_{p,1}}\|g\|_{{B}^{s_2}_{p,1}},\quad s_1,s_2\leq \frac{d}{p},\  s_1+s_2>d\max\{0,\frac{2}{p}-1\}
\end{align}
to estimate the term $-\partial_{x}(1-\partial_{xx})^{-1}\Big(u^2+\frac{u^2_x}{2}\Big)$ in CH equation.
That is why one needs the condition $s>\max\{\frac{3}{2},1+\frac{1}{p}\}$ ($s=1+\frac{1}{p},~p\in[1,2]$). To overcome the difficulty, inspired by Linares, Ponce and Thomas \cite{lps}, we use the Lagrange coordinate transformation and small time conditions to investigate the uniqueness for the Camassa-Holm type equations. Indeed, suppose that the time $T>0$ small enough, we can ensure that the characteristic $y(t,\xi)$ in a small time interval $[0,T]$ is a homeomorphism, and then we will obtain the uniqueness without using \eqref{mosher}, see the proof of Theorem \ref{lagrabstr} in Section 3.

To establish the well-posedness for the  Camassa-Holm type equations, we first consider the following Cauchy problem for a general abstract equation
\begin{equation}\label{abstract}
\left\{\begin{array}{ll}
\partial_{t}u+A(u)\partial_{x}u=F(u),&\quad t>0,\quad x\in\mathbb{R},\\
u(t,x)|_{t=0}=u_0(x),&\quad x\in\mathbb{R},
\end{array}\right.
\end{equation}
where $A(u)$ is a polynomial of $u$ and $F$ is called a `good operator' such that for any $\varphi\in \mathcal{C}^{\infty}_0(\mathbb{R})$ and any $\epsilon>0$ small enough, the following fact holds
$$\text{if}\quad u_n\varphi\rightarrow u\varphi\quad \text{in}\quad B^{1+\frac{1}{p}-\epsilon}_{p,1},\quad \text{then}\quad  \langle F(u_n),\varphi\rangle\longrightarrow\langle F(u),\varphi\rangle.$$
This definition is reasonable for the Camassa-Holm type equations. For example, it's easy to prove that $-\partial_{x}(1-\partial_{xx})^{-1}\Big(u^2+\frac{1}{2}u^2_x\Big)$ is a `good operator' by an approximation argument, owing to $\mathcal{C}^{\infty}_0(\mathbb{R})$ is dense in $\mathcal{S}(\mathbb{R})$.

The associated Lagrangian scale of \eqref{abstract} is the following initial valve problem
\begin{equation}\label{ODE}
\left\{\begin{array}{ll}
\frac{{\ud}y}{{\ud}t}=A(u)\big(t,y(t,\xi)\big),&\quad t>0,\quad \xi\in\mathbb{R},\\
y(0,\xi)=\xi,&\quad \xi\in\mathbb{R}.
\end{array}\right.
\end{equation}
Introduce the new variable $U(t,\xi)=u\big(t,y(t,\xi)\big)$. Then, \eqref{abstract} becomes
\begin{equation}\label{lagrange}
\left\{\begin{array}{ll}
U_t=\Big(F(u)\Big)(t,y(t,\xi)):=\widetilde{F}(U,y),&\quad t>0,\quad \xi\in\mathbb{R},\\
U(t,\xi)|_{t=0}=U_0(\xi)=u_0(\xi),&\quad \xi\in\mathbb{R}.
\end{array}\right.
\end{equation}

Before giving our main results, let's introduce two key abstract theorems:
\begin{theo}\label{lagrabstr}
Let $u_0\in B^{1+\frac{1}{p}}_{p,1}$ with $p\in[1,\infty)$, $k\in\mathbb{N}^{+}$ and $F$ is a good `operator'. Consider the following conditions for $F,\ \widetilde{F}:$
\begin{align}
&\|F(u)\|_{B^{1+\frac{1}{p}}_{p,1}}\leq C\Big(\|u\|^{k+1}_{B^{1+\frac{1}{p}}_{p,1}}+1\Big);\label{fkl}\\
&\|\widetilde{F}(U,y)-\widetilde{F}(\bar{U},\bar{y})\|_{W^{1,\infty}\cap W^{1,p}}\leq C_{u_0}\Big(\|U-\bar{U}\|_{W^{1,\infty}\cap W^{1,p}}+\|y-\bar{y}\|_{W^{1,\infty}\cap W^{1,p}}\Big);\label{wideFkl}\\
&\|F(u)-F(\bar{u})\|_{B^{1+\frac{1}{p}}_{p,1}}\leq C_{u_0}\|u-\bar{u}\|_{B^{1+\frac{1}{p}}_{p,1}}.\label{Fkl}
\end{align}
Then, there exists a time $T>0$ such that
\begin{itemize}
\item [\rm{(1)}] Existence: If \eqref{fkl} holds, then \eqref{abstract} has a solution $u\in E^p_T:=\mathcal{C}\big([0,T];B^{1+\frac{1}{p}}_{p,1}\big)\cap \mathcal{C}^{1}\big([0,T];B^{\frac{1}{p}}_{p,1}\big)$;
\item [\rm{(2)}] Uniqueness: If \eqref{fkl} and \eqref{wideFkl} hold, then the solution of \eqref{abstract} is unique;
\item [\rm{(3)}] Continuous dependence: If \eqref{fkl}--\eqref{Fkl} hold, then the solution map is continuous from any bounded sunset of $ B^{1+\frac{1}{p}}_{p,1}$ to $ \mathcal{C}\big([0,T];B^{1+\frac{1}{p}}_{p,1}\big)$.
\end{itemize}
That is, the problem \eqref{abstract} is locally well-posed in the sense of Hadamard.
\end{theo}

Similarly, for the two-component Camassa-Holm type equations, we can investigate the following abstract equations :
\begin{equation}\label{abstr}
	\left\{\begin{array}{ll}
		\partial_{t}u+A(u)\partial_{x}u=F\big(u,v\big),&\quad t>0,\quad x\in\mathbb{R},\\
		\partial_{t}v+A(u)\partial_{x}v=-\partial_{x}[A(u)]\cdot h(v),&\quad t>0,\quad x\in\mathbb{R},\\
		u(t,x)|_{t=0}=u_0(x),\ v(t,x)|_{t=0}=v_0(x)&\quad x\in\mathbb{R},
	\end{array}\right.
\end{equation}
where $A,\ F$ satisfy the above conditions in Theorem \ref{lagrabstr} and $h$ is a affine function (i.e., $h(v)=av+b$ for some $a,b\in\mathbb{R}$). Define the new variables $U(t,\xi)=u\big(t,y(t,\xi)\big),V(t,\xi)=v\big(t,y(t,\xi)\big)$. Combining the above theorem with \eqref{ODE}, we can obtain the other abstract theorem.
\begin{theo}\label{lagrabstr2}
Let $(u_0,v_{0})\in B^{1+\frac{1}{p}}_{p,1}\times B^{\frac{1}{p}}_{p,1}$ with $p\in[1,\infty)$, $k\in\mathbb{N}^{+}$, $F(u,v)$ is a good `operator'. Assume that the operators $F,\ \widetilde{F}$ satisfy the following conditions:
\begin{align}
&\|F(u,v)\|_{B^{1+\frac{1}{p}}_{p,1}}\leq C\Big(\|u\|^{k+1}_{B^{1+\frac{1}{p}}_{p,1}}+\|v\|^{k+1}_{B^{\frac{1}{p}}_{p,1}}+1\Big);\label{fkl1}\\
&\|\widetilde{F}(U,V,y)-\widetilde{F}(\bar{U},\bar{V},\bar{y})\|_{W^{1,\infty}\cap W^{1,p}}\leq C_{u_0}\Big(\|U-\bar{U}\|_{W^{1,\infty}\cap W^{1,p}}\notag\\
&\qquad\qquad\qquad\qquad\qquad\qquad\qquad\qquad+\|V-\bar{V}\|_{L^{\infty}\cap L^{p}}+\|y-\bar{y}\|_{W^{1,\infty}\cap W^{1,p}}\Big);\label{wideFkl1}\\
&\|F(u,v)-F(\bar{u},\bar{v})\|_{B^{1+\frac{1}{p}}_{p,1}}\leq C_{u_0}\Big(\|u-\bar{u}\|_{B^{1+\frac{1}{p}}_{p,1}}+\|v-\bar{v}\|_{B^{\frac{1}{p}}_{p,1}}\Big).\label{Fkl1}
\end{align}
Then, there exists a time $T>0$ such that
\begin{itemize}
\item [\rm{(1)}] Existence: If  \eqref{fkl1} holds, then \eqref{abstr} has a solution $(u,v)$ in $\mathcal{C}\big([0,T];B^{1+\frac{1}{p}}_{p,1}\big)\times \mathcal{C}\big([0,T];B^{\frac{1}{p}}_{p,1}\big)$;
\item [\rm{(2)}] Uniqueness: If \eqref{fkl1} and \eqref{wideFkl1} hold, then the solution of \eqref{abstr} is unique;
\item [\rm{(3)}] Continuous dependence: If \eqref{fkl1}--\eqref{Fkl1} hold, then the solution map is continuous from any bounded subset of $B^{1+\frac{1}{p}}_{p,1}\times B^{\frac{1}{p}}_{p,1}$ to $\mathcal{C}\big([0,T];B^{1+\frac{1}{p}}_{p,1}\big)\times \mathcal{C}\big([0,T];B^{\frac{1}{p}}_{p,1}\big)$.
\end{itemize}
That is, the problem \eqref{abstr}  is locally well-posed in the sense of Hadamard.
\end{theo}

Applying Theorem \ref{lagrabstr} and \ref{lagrabstr2}, we can obtain our main results (more details see Section 4).
\begin{theo}\label{CH}
Let $u_0\in B^{1+\frac{1}{p}}_{p,1}$ with $p\in[1,\infty)$. Then there exists a time $T>0$ such that the CH equation with the initial data $u_0$ is locally well-posed in the sense of Hadamard.

\end{theo}
\begin{theo}\label{N}
Let $u_0\in B^{1+\frac{1}{p}}_{p,1}$ with $p\in[1,\infty)$. Then there exists a time $T>0$ such that the Novikov equation with the initial data $u_{0}$ is locally well-posed in the sense of Hadamard.
\end{theo}

\begin{theo}\label{2CH}
Let $(u_0,\rho_0-1)\in B^{1+\frac{1}{p}}_{p,1}\times B^{\frac{1}{p}}_{p,1}$ with $p\in[1,\infty)$. Then there exists a time $T>0$ such that the two-component Camassa-Holm system with the initial data $(u_0,\rho_0-1)$ is locally well-posed in the sense of Hadamard.
\end{theo}
\begin{rema}
By the new results, we promote the index $p$ to $[1,+\infty)$ for the well-posedness of the Camassa-Holm type equations, which improves the previous results in \cite{d1,d2,liy}, see the following
\begin{align*}
\underset{\text{recent~ results~in\cite{d1,d2,liy}}}{\overset{\text{well-posed}}{B^{1+\frac{1}{p}}_{p,1}(\mathbb{R}),~p\in [1,2]~}}~\Longrightarrow\underset{\text{our results}}{\overset{\text{well-posed}}{B^{1+\frac{1}{p}}_{p,1}(\mathbb{R}),~p\in [1,\infty).}}
\end{align*}
Notice that Guo et al. \cite{glmy} established
the ill-posedness for the Camassa-Holm type equations in $B^{1+\frac 1 p}_{p,r}$ with $p\in[1,+\infty],\ r\in(1,+\infty].$ This implies that $B^{1+\frac 1 p}_{p,1},~p\in[1,+\infty)$ are the critical Besov spaces for the well-posedness of the Camassa-Holm type equations. However, when $p=\infty$, since $B^{0}_{\infty,1}$ is not an algebra, the estimation of $\|\partial_x(1-\partial_{xx})^{-1}(\frac{u_{0x}^2}{2})\|_{B^{1}_{\infty,1}}$ is hard to deal with. We don't know whether the Cauchy problem of the Camasa-Holm type equations is well-posed or not in $B^{1}_{\infty,1}$, which is the only problem of the well-posedness for the Camassa-Holm type equations in Besov spaces. In fact, if $u_0\in B^{1+\epsilon}_{\infty,1}$ with any $\epsilon>0$ ($B^{\epsilon}_{\infty,1}$ is an algebra), similar to the proof of Theorem \ref{lagrabstr}, one can obtain the local well-posedness for the Camassa-Holm type equations.
\end{rema}
\begin{rema}
For $u_0\in H^{1}\cap W^{1,\infty}$, Linares, Ponce and Thomas \cite{lps} established the well-posedness for the CH equation in $\mathcal{C}([0,T];H^{1})\cap L^{\infty}\big(0,T;W^{1,\infty}\big)$ which contains the peakons  $u(t,x)=ce^{-|x-ct|}$. Moreover, they proved that the continuity of the map from $W^{1,\infty}$ into $\mathcal{C}([0,T];W^{1,\infty})$ fails in any time interval $[0,T]$ for any $T>0$ by taking the peakons as an example. But the spaces $B^{1+\frac 1 p}_{p,1}$ we consider in this paper doesn't contain the peakons, and thus we obtain the local well-posedness in $C([0,T];B^{1+\frac 1 p}_{p,1})$, which implies that our results are substantially different from theirs.
\end{rema}

Finally, noting again that the Camassa-Holm type equations is ill-posedness in $B^{1+\frac 1 p}_{p,r}$ with $p\in[1,+\infty],\ r\in(1,+\infty]$ by constructing a special initial data $u_0\in B^{1+\frac 1 p}_{p,r}$ but $u_0\notin W^{1,\infty}$ (see \cite{glmy}), a nature problem is that whether it is well-posed or not in $B^{1+\frac 1 p}_{p,r}$ with any initial data in $ B^{1+\frac 1 p}_{p,r}\cap W^{1,\infty}$? Inspired by the idea of Linares, Ponce and Thomas \cite{lps}, we will give a positive reply in the end of this paper, see Theorem \ref{general}.

The rest of our paper is as follows. In the second section, we introduce some preliminaries which will be used in the sequel. In the third section, we give the proof of Theorem \ref{lagrabstr}--\ref{lagrabstr2} by using the Lagrangian coordinate transformation. In the last section, by applying the abstract theorems, we establish the local well-posedness for the Cauchy problem of several Camassa-Holm type (CH, Novikov, two-component Camassa-Holm) equations in critical Besov spaces $B^{1+\frac{1}{p}}_{p,1}$ with $p\in[1,+\infty)$, which is a new result.

\section{Preliminaries}
\par
In this section, we first recall some basic properties on the Littlewood-Paley theory, which can be found in \cite{book}.

Let $\chi$ and $\varphi$ be a radical, smooth, and valued in the interval $[0,1]$, belonging respectively to $\mathcal{D}(\mathbf{B})$ and $\mathcal{D}(\mathbf{C})$, where $\mathbf{B}=\{\xi\in\mathbb{R}^d:|\xi|\leq\frac 4 3\},\ \mathbf{C}=\{\xi\in\mathbb{R}^d:\frac 3 4\leq|\xi|\leq\frac 8 3\}$.
Denote $\mathcal{F}$ by the Fourier transform and $\mathcal{F}^{-1}$ by its inverse.
For any $u\in\mathcal{S}'(\mathbb{R}^d)$,  all $j\in\mathbb{Z}$, define
$\Delta_j u=0$ for $j\leq -2$; $\Delta_{-1} u=\mathcal{F}^{-1}(\chi\mathcal{F}u)$; $\Delta_j u=\mathcal{F}^{-1}(\varphi(2^{-j}\cdot)\mathcal{F}u)$ for $j\geq 0$; and $S_j u=\sum_{j'<j}\Delta_{j'}u$.

Let $s\in\mathbb{R},\ 1\leq p,r\leq\infty.$ The nonhomogeneous Besov space $B^s_{p,r}(\mathbb{R}^d)$ is defined by
$$  B^s_{p,r}=B^s_{p,r}(\mathbb{R}^d)=\Big\{u\in S'(\mathbb{R}^d):\|u\|_{B^s_{p,r}}=\big\|(2^{js}\|\Delta_j u\|_{L^p})_j \big\|_{l^r(\mathbb{Z})}<\infty\Big\}.$$

The nonhomogeneous Sobolev space is defined by
$$H^{s}=H^{s}(\mathbb{R}^d)=\Big\{u\in S'(\mathbb{R}^d):\ u\in L^2_{loc}(\mathbb{R}^d),\ \|u\|^2_{H^s}=\int_{\mathbb{R}^d}(1+|\xi|^2)^s|\mathcal{F}u(\xi)|^2{\ud}\xi<\infty\Big\}. $$

The nonhomogeneous Bony's decomposition is defined by
$uv=T_{u}v+T_{v}u+R(u,v)$ with
$$T_{u}v=\sum_{j}S_{j-1}u\Delta_{j}v,\ \ R(u,v)=\sum_{j}\sum_{|j'-j|\leq 1}\Delta_{j}u\Delta_{j'}v.$$

Naturally, we introduce some properties about Besov spaces. For more details, see \cite{book}.
\begin{prop}\label{Besov}\cite{book}
Let $s\in\mathbb{R},\ 1\leq p,p_1,p_2,r,r_1,r_2\leq\infty.$  \\
{\rm(1)} $B^s_{p,r}$ is a Banach space, and is continuously embedded in $\mathcal{S}'$. \\
{\rm(2)} If $r<\infty$, then $\lim\limits_{j\rightarrow\infty}\|S_j u-u\|_{B^s_{p,r}}=0$. If $p,r<\infty$, then $\mathcal{C}_0^{\infty}$ is dense in $B^s_{p,r}$. \\
{\rm(3)} If $p_1\leq p_2$ and $r_1\leq r_2$, then $ B^s_{p_1,r_1}\hookrightarrow B^{s-d(\frac 1 {p_1}-\frac 1 {p_2})}_{p_2,r_2}. $
If $s_1<s_2$, then the embedding $B^{s_2}_{p,r_2}\hookrightarrow B^{s_1}_{p,r_1}$ is locally compact. \\
{\rm(4)} $B^s_{p,r}\hookrightarrow L^{\infty} \Leftrightarrow s>\frac d p\ \text{or}\ s=\frac d p,\ r=1$. \\
{\rm(5)} Fatou property: if $(u_n)_{n\in\mathbb{N}}$ is a bounded sequence in $B^s_{p,r}$, then an element $u\in B^s_{p,r}$ and a subsequence $(u_{n_k})_{k\in\mathbb{N}}$ exist such that
$$ \lim_{k\rightarrow\infty}u_{n_k}=u\ \text{in}\ \mathcal{S}'\quad \text{and}\quad \|u\|_{B^s_{p,r}}\leq C\liminf_{k\rightarrow\infty}\|u_{n_k}\|_{B^s_{p,r}}. $$
{\rm(6)} Let $m\in\mathbb{R}$ and $f$ be a $S^m$-mutiplier (i.e. f is a smooth function and satisfies that $\forall\ \alpha\in\mathbb{N}^d$,
$\exists\ C=C(\alpha)$ such that $|\partial^{\alpha}f(\xi)|\leq C(1+|\xi|)^{m-|\alpha|},\ \forall\ \xi\in\mathbb{R}^d)$.
Then the operator $f(D)=\mathcal{F}^{-1}(f\mathcal{F})$ is continuous from $B^s_{p,r}$ to $B^{s-m}_{p,r}$.
\end{prop}
\begin{prop}\cite{book}
Let $s\in\mathbb{R},\ 1\leq p,r\leq\infty.$
\begin{equation*}\left\{
\begin{array}{l}
B^s_{p,r}\times B^{-s}_{p',r'}\longrightarrow\mathbb{R},  \\
(u,\phi)\longmapsto \sum\limits_{|j-j'|\leq 1}\langle \Delta_j u,\Delta_{j'}\phi\rangle,
\end{array}\right.
\end{equation*}
defines a continuous bilinear functional on $B^s_{p,r}\times B^{-s}_{p',r'}$. Denote by $Q^{-s}_{p',r'}$ the set of functions $\phi$ in $\mathcal{S}'$ such that $\|\phi\|_{B^{-s}_{p',r'}}\leq 1$. If $u$ is in $\mathcal{S}'$, then we have
$$\|u\|_{B^s_{p,r}}\leq C\sup_{\phi\in Q^{-s}_{p',r'}}\langle u,\phi\rangle.$$
\end{prop}

The useful interpolation inequalities are given as follows.
\begin{prop}\label{prop}\cite{book}
{\rm(1)} If $s_1<s_2$, $\lambda\in (0,1)$ and $(p,r)\in[1,\infty]^2$, then we have
\begin{align*}
&\|u\|_{B^{\lambda s_1+(1-\lambda)s_2}_{p,r}}\leq \|u\|_{B^{s_1}_{p,r}}^{\lambda}\|u\|_{B^{s_2}_{p,r}}^{1-\lambda},\\
&\|u\|_{B^{\lambda s_1+(1-\lambda)s_2}_{p,1}}\leq\frac{C}{s_2-s_1}\Big(\frac{1}{\lambda}+\frac{1}{1-\lambda}\Big) \|u\|_{B^{s_1}_{p,\infty}}^{\lambda}\|u\|_{B^{s_2}_{p,\infty}}^{1-\lambda}.
\end{align*}
{\rm(2)} If $s\in\mathbb{R},\ 1\leq p\leq\infty,\ \varepsilon>0$, a constant $C=C(\varepsilon)$ exists such that
$$ \|u\|_{B^s_{p,1}}\leq C\|u\|_{B^s_{p,\infty}}\ln\Big(e+\frac {\|u\|_{B^{s+\varepsilon}_{p,\infty}}}{\|u\|_{B^s_{p,\infty}}}\Big). $$
\end{prop}

We now give the 1-D Moser-type estimates which we will use in the following.
\begin{lemm}\label{product}\cite{book,liy}
The following estimates hold:\\
{\rm(1)} For any $s>0$ and any $p,\ r$ in $[1,\infty]$, the space $L^{\infty} \cap B^s_{p,r}$ is an algebra, and a constant $C=C(s)$ exists such that
\begin{align*}
&\|uv\|_{B^s_{p,r}}\leq C(\|u\|_{L^{\infty}}\|v\|_{B^s_{p,r}}+\|u\|_{B^s_{p,r}}\|v\|_{L^{\infty}}), \\
&\|u\partial_{x}v\|_{B^s_{p,r}}\leq C(\|u\|_{B^{s+1}_{p,r}}\|v\|_{L^{\infty}}+\|u\|_{L^{\infty}}\|\partial_{x}v\|_{B^s_{p,r}}).
\end{align*}
{\rm(2)} If $1\leq p,r\leq \infty,\ s_1\leq s_2,\ s_2>\frac{1}{p} (s_2 \geq \frac{1}{p}\ \text{if}\ r=1)$ and $s_1+s_2>\max(0, \frac{2}{p}-1)$, there exists $C=C(s_1,s_2,p,r)$ such that
$$ \|uv\|_{B^{s_1}_{p,r}}\leq C\|u\|_{B^{s_1}_{p,r}}\|v\|_{B^{s_2}_{p,r}}. $$
\end{lemm}

Here is the Gronwall lemma.
\begin{lemm}\label{gwl}\cite{book}
Let $m(t),\ a(t)\in \mathcal{C}^1([0,T]),\ m(t),\ a(t)>0$. Let $b(t)$ is a continuous function on $[0,T]$. Suppose that, for all $t\in [0,T]$,
$$\frac{1}{2}\frac{{\ud}}{{\ud}t}m^2(t)\leq b(t)m^2(t)+a(t)m(t).$$
Then for any time $t$ in $[0,T]$, we have
$$m(t)\leq m(0)\exp\int_0^t b(\tau){\ud}\tau+\int_0^t a(\tau)\exp\big(\int_{\tau}^tb(t'){\ud}t'\big){\ud}\tau.$$
\end{lemm}

In the paper, we also need some estimates for the following 1-D transport equation:
\begin{equation}\label{transport}
\left\{\begin{array}{l}
f_t+v\partial_{x}f=g,\ x\in\mathbb{R},\ t>0, \\
f(0,x)=f_0(x),\ x\in\mathbb{R}.
\end{array}\right.
\end{equation}

\begin{lemm}\cite{book,liy}\label{existence}
Let $1\leq p\leq\infty,\ 1\leq r\leq\infty,\ \theta> -\min(\frac 1 {p}, \frac 1 {p'})$. Let $f_0\in B^{\theta}_{p,r}$, $g\in L^1(0,T;B^{\theta}_{p,r})$, and $v\in L^\rho(0,T;B^{-M}_{\infty,\infty})$ for some $\rho>1$ and $M>0$ such that
\begin{align*}
\begin{array}{ll}
\partial_{x}v\in L^1(0,T;B^{\frac 1 {p}}_{p,\infty}\cap L^{\infty}), &\ \text{if}\ \theta<1+\frac 1 {p}, \\
\partial_{x}v\in L^1(0,T;B^{\theta}_{p,r}),\ &\text{if}\ \theta=1+\frac{1}{p},\ r>1, \\
\partial_{x}v\in L^1(0,T;B^{\theta-1}_{p,r}), &\ \text{if}\ \theta>1+\frac{1}{p}\ (or\ \theta=1+\frac 1 {p},\ r=1).
\end{array}	
\end{align*}
Then the problem \eqref{transport} has a unique solution $f$ in \\
-the space $\mathcal{C}\big([0,T];B^{\theta}_{p,r}\big)$, if $r<\infty$, \\
-the space $\Big(\bigcap_{{\theta}'<\theta}\mathcal{C}([0,T];B^{{\theta}'}_{p,\infty})\Big)\bigcap \mathcal{C}_w\big([0,T];B^{\theta}_{p,\infty}\big)$, if $r=\infty$.
\end{lemm}

\begin{lemm}\label{priori estimate}\cite{book,liy}
Let $1\leq p,r\leq\infty,\ \theta>-min(\frac{1}{p},\frac{1}{p'}).$
There exists a constant $C$ such that for all solutions $f\in L^{\infty}(0,T;B^{\theta}_{p,r})$ of \eqref{transport} with initial data $f_0$ in $B^{\theta}_{p,r}$ and $g$ in $L^1(0,T;B^{\theta}_{p,r})$,
$$ \|f(t)\|_{B^{\theta}_{p,r}}\leq \|f_0\|_{B^{\theta}_{p,r}}+\int_0^t\|g(t')\|_{B^{\theta}_{p,r}}{\ud}t'+\int_0^t V'(t')\|f(t')\|_{B^{\theta}_{p,r}}{\ud}t' $$
or
$$ \|f(t)\|_{B^{\theta}_{p,r}}\leq e^{CV(t)}\Big(\|f_0\|_{B^{\theta}_{p,r}}+\int_0^t e^{-CV(t')}\|g(t')\|_{B^{\theta}_{p,r}}{\ud}t'\Big) $$
with
\begin{equation*}
V'(t)=\left\{\begin{array}{ll}
\|\partial_{x}v(t)\|_{B^{\frac 1 p}_{p,\infty}\cap L^{\infty}},\ &\text{if}\ \theta<1+\frac{1}{p}, \\
\|\partial_{x}v(t)\|_{B^{\theta}_{p,r}},\ &\text{if}\ \theta=1+\frac{1}{p},\ r>1, \\
\|\partial_{x}v(t)\|_{B^{\theta-1}_{p,r}},\ &\text{if}\ \theta>1+\frac{1}{p}\ (\text{or}\ \theta=1+\frac{1}{p},\ r=1).
\end{array}\right.
\end{equation*}
If $\theta>0$, then there exists a constant $C=C(p,r,\theta)$ such that the following statement holds
\begin{align*}
\|f(t)\|_{B^{\theta}_{p,r}}\leq \|f_0\|_{B^{\theta}_{p,r}}+\int_0^t\|g(\tau)\|_{B^{\theta}_{p,r}}{\ud}\tau+C\int_0^t \Big(\|f(\tau)\|_{B^{\theta}_{p,r}}\|\partial_{x}v(\tau)\|_{L^{\infty}}+\|\partial_{x}v(\tau)\|_{B^{\theta-1}_{p,r}}\|\partial_{x}f(\tau)\|_{L^{\infty}}\Big){\ud}\tau.	
\end{align*}
In particular, if $f=av+b,\ a,\ b\in\mathbb{R},$ then for all $\theta>0,$ $V'(t)=\|\partial_{x}v(t)\|_{L^{\infty}}.$
\end{lemm}
\begin{lemm}\label{continuous}\cite{book,liy}
Let $1\leq p<\infty$. Define $\overline{\mathbb{N}}=\mathbb{N}\cup\{\infty\}$. Suppose $f\in L^{1}\big(0,T;B^{\frac{1}{p}}_{p,1}\big)$ and $a_0\in B^{\frac{1}{p}}_{p,1}$. For $n\in\overline{\mathbb{N}}$, denote by $a^n\in \mathcal{C}\big([0,T];B^{\frac{1}{p}}_{p,1}\big)$ the solution of
\begin{equation}\label{an}
\left\{\begin{array}{l}
\partial_{t}a^n+A^n\partial_{x}a^n=f,\\
a^n(0,x)=a_0(x).
\end{array}\right.
\end{equation}
Assume for some $\beta\in L^{1}(0,T)$, $\sup\limits_{n\in\overline{\mathbb{N}}}\|A^n\|_{B^{1+\frac{1}{p}}_{p,1}}\leq \beta(t)$.
If $A^{n}$ converges to $A^{\infty}$ in $L^{1}\big(0,T;B^{\frac{1}{p}}_{p,1}\big)$, then the sequence $\{a^n\}_{n\in\mathbb{N}}$ converges to $a^\infty$ in $\mathcal{C}\big([0,T];B^{\frac{1}{p}}_{p,1}\big)$.
\end{lemm}

\section{The proof of abstract theorems}
\par
In this section, we give the proof of Theorem \ref{lagrabstr} and Theorem \ref{lagrabstr2}.
\begin{proof}[\rm{\textbf{The proof of Theorem \ref{lagrabstr}:}}] For the sake of simplicity, we set $A(u)=u$ in the following proof (other cases are similar).

\textbf{Step 1. Existence.} 

We firstly set $u^0\triangleq0$. Define a sequence $\{u^n\}_{n\in\mathbb{N}}$ of smooth functions by solving the following linear transport equations:
\begin{equation}\label{eq2}
    \left\{\begin{array}{l}
    u_t^{n+1}+u^nu_x^{n+1}=F(u^n), \\
    u^{n+1}|_{t=0}=u_0,
    \end{array}\right.
\end{equation}
Combining Lemma \ref{existence} with \eqref{fkl}, we obtain a unique solution $u^{n+1}$ of \eqref{eq2} in $E^p_T$.
Moreover, by Lemma \ref{priori estimate} and \eqref{fkl}, we have
\begin{equation}\label{fuzhiineq23}
    \|u^{n+1}(t)\|_{B^{1+\frac{1}{p}}_{p,1}}\leq e^{C\int_0^t \|\partial_{x}u^n\|_{B^{\frac{1}{p}}_{p,1}}{\ud}t'}
    \Big(\|u_0\|_{B^{1+\frac{1}{p}}_{p,1}}+\int_0^t e^{-C\int_0^{t'} \|\partial_{x}u^n\|_{B^{\frac{1}{p}}_{p,1}}{\ud}t''}\big(
    \|u^n\|^{k+1}_{B^{1+\frac{1}{p}}_{p,1}}+1\big){\ud}t'\Big),
\end{equation}
Fix a time $T$ such that it satifies $0<T<\frac{1}{C(\|u_0\|^{k+1}_{B^{\frac{1}{p}+1}_{p,1}}+1)}$. Similar to the proof of Theorem 3.21 in \cite{book}, we obtain that
$$\|u^{n+1}(t)\|_{L^\infty\big(0,T;B^{1+\frac{1}{p}}_{p,1}\big)}\leq C_{u_0}.$$
Therefore, $\{u^n\}_{n\in \mathbb{N}}$ is bounded in $L^\infty\big(0,T;B^{1+\frac{1}{p}}_{p,1}\big)$.

Then, we will use the compactness method for the approximating sequence $\{u_n\}_{n\in\mathbb{N}}$ to get a solution $u$ of \eqref{abstract}. Since $u_n$ is uniformly bounded in $L^\infty\big(0,T;B^{1+\frac{1}{p}}_{p,1}\big)$, we can deduce from \eqref{fkl} that $\partial_{t}u_n$ is uniformly bounded in $L^\infty\big(0,T;B^{\frac{1}{p}}_{p,1}\big)$. Thus,
\begin{align*}
u_n\ \text{is uniformly bounded in}\ \mathcal{C}\big([0,T];B^{1+\frac{1}{p}}_{p,1}\big)\cap \mathcal{C}^{\frac{1}{2}}\big([0,T];B^{\frac{1}{p}}_{p,1}\big).	
\end{align*}
Let $(\phi_j)_{j\in\mathbb{N}}$ be a sequence of smooth functions with value in $[0,1]$ supported in the ball $B(0,j+1)$ and equal to 1 on $B(0,j)$. Notice that the map $z\mapsto\phi_jz$ is compact from $B^{1+\frac{1}{p}}_{p,1}$ to $B^{\frac{1}{p}}_{p,1}$ by Theorem 2.94 in \cite{book}. Taking advantage of Ascoli's theorem and Cantor's diagonal process, there exists some function $u_j$ such that for any $j\in\mathbb{N}$, $\phi_ju_n$ tends to $u_j$. From that, we can easily deduce that there exists some function $u$ such that for all $\phi\in\mathcal{D}$, $\phi u_n$ tends to $\phi u$ in $\mathcal{C}\big([0,T];B^{\frac{1}{p}}_{p,1}\big)$. Combining the uniform boundness of $u_n$ and the Fatou property for Besov spaces, we realiy obtain that $u\in L^\infty\big(0,T;B^{1+\frac{1}{p}}_{p,1}\big)$. By virtue of the interpolation inequality, we have $\phi u_n$ tends to $\phi u$ in $\mathcal{C}\big([0,T];B^{1+\frac{1}{p}-\varepsilon}_{p,1}\big)$ for any $\varepsilon>0$. Next, since $F(u)$ is a good `operator', it is a routine process to prove that $u$ satifies Eq. \eqref{abstract}. Thanks to the right side of the Eq. \eqref{abstract}, we get $\partial_{t}u\in \mathcal{C}\big([0,T];B^{\frac{1}{p}}_{p,1}\big)$. In sum, we obtain $u$ satisfies  \eqref{abstract} and belongs to $E^p_T$.

\textbf{Step 2. Uniqueness.}
	
Define $U(t,\xi)=u(t,y(t,\xi))$. Thus, $U_{\xi}(t,\xi)=u_x(t,y(t,\xi))y_{\xi}(t,y(t,\xi))$. Owing to \eqref{abstract} and \eqref{ODE}, we can get
\begin{align}
&y(t,\xi)=\xi+\int_0^tU{\ud}\tau,\label{yn}\\
&y_\xi(t,\xi)=1+\int_0^tU_\xi {\ud}\tau,\label{ynxi}\\
&U_t(t,\xi)=\widetilde{F}(U,y),\label{Un}\\
&U_{t\xi}(t,\xi)=\left(\widetilde{F}(U,y)\right)_{\xi}.\label{Unxi}
\end{align}
Since $u$ is uniformly bounded in $\mathcal{C}\big([0,T];B^{1+\frac{1}{p}}_{p,1}\big)\hookrightarrow \mathcal{C}\big([0,T];W^{1,p}\cap W^{1,\infty}\big)$, we can easily deduce that $y_{\xi}$ is bounded in $L^\infty(0,T;L^\infty)$ by the Gronwall inequality. So $U(t,\xi)$ is bounded in $L^\infty(0,T;W^{1,\infty})$. Moreover, by \eqref{ynxi} we deduce that $\frac{1}{2}\leq y_\xi\leq C_{u_0}$ for $T>0$ small enough. This implies the flow map $y(t,\xi)$ is one-order diffeomorphism, which seems to be a more simple operation than \cite{lps}. In this case, we can easily prove $U(t,\xi)\in L^\infty(0,T;W^{1,p})$. Indeed,
\begin{align*}
	&\|U\|_{L^p}^p=\int_{-\infty}^{+\infty}|U(t,\xi)|^p{\ud}\xi=\int_{-\infty}^{+\infty}|u(t,y(t,\xi))|^p\frac{1}{y_\xi}{\ud}y\leq2\|u\|_{L^p}^p\leq C;\\
	&\|U_\xi\|_{L^p}^p=\int_{-\infty}^{+\infty}|U_\xi(t,\xi)|^p{\ud}\xi=\int_{-\infty}^{+\infty}|u_x(t,y(t,\xi))|^p{y_\xi}^{p-1}{\ud}y\leq C_{u_{0}}^{p-1}\|u_x\|_{L^p}^p\leq C.
\end{align*}
Thus, we obtain that $U(t,\xi)\in L^\infty(0,T;W^{1,p}\cap W^{1,\infty})$, $y(t,\xi)-\xi\in L^\infty(0,T;W^{1,p}\cap W^{1,\infty})$ and $\frac{1}{2}\leq y_\xi(t,\xi)\leq C_{u_0}$ for any $t\in[0,T]$.

Now we prove the uniqueness. Suppose that $u_1,u_2$ are two solutions to \eqref{abstract}, then $U_{i}(t,\xi)=u_{i}(t,y_{i}(t,\xi))$ satifies \eqref{Un} and \eqref{Unxi} for $i=1,2$. Hence, \eqref{wideFkl}, \eqref{Un}, \eqref{Unxi} together with the Growall inequality yield
\begin{align*}
&\quad\|U_{1}-U_{2}\|_{W^{1,\infty}\cap W^{1,p}}+\|y_{1}-y_{2}\|_{W^{1,\infty}\cap W^{1,p}}\notag\\
&\leq C(\|U_{1}(0)-U_{2}(0)\|_{W^{1,\infty}\cap W^{1,p}}+\|y_{1}(0)-y_{2}(0)\|_{W^{1,\infty}\cap W^{1,p}})\notag\\
&\quad +C\int_{0}^{T}\|\widetilde{F}(U_1,y_1)-\widetilde{F}(U_2,y_2)\|_{W^{1,\infty}\cap W^{1,p}}+\|U_1-U_2\|_{W^{1,\infty}\cap W^{1,p}}{\ud}t\notag\\
&\leq C(\|U_{1}(0)-U_{2}(0)\|_{W^{1,\infty}\cap W^{1,p}}+0)+C\int_{0}^{T}\Big(\|U_1-U_2\|_{W^{1,\infty}\cap W^{1,p}}+\|y_1-y_2\|_{W^{1,\infty}\cap W^{1,p}}\Big){\ud}t\notag\\
&\leq C\|u_{1}(0)-u_{2}(0)\|_{B^{1+\frac{1}{p}}_{p,1}},
\end{align*}
where $y_1(0)=y_2(0)=\xi$, $U_i(0)=u_i(0),~i=1,2$. It follows that
\begin{align*}
	\|u_{1}-u_{2}\|_{L^p}\leq& C\|u_{1}\circ y_{1}-u_{2}\circ y_{1}\|_{L^p}\\
	\leq& C\|u_{1}\circ y_{1}-u_{2}\circ y_{2}+u_{2}\circ y_{2}-u_{2}\circ y_{1}\|_{L^p}\\
	\leq& C\|U_{1}-U_{2}\|_{L^p}+C\|u_{2x}\|_{L^\infty}\|y_{1}-y_{2}\|_{L^p}\\
	\leq& C\|u_{1}(0)-u_{2}(0)\|_{B^{1+\frac{1}{p}}_{p,1}}.
\end{align*}
By the embedding $L^p\hookrightarrow B^0_{p,\infty}$, we get
\begin{align*}
	\|u_{1}-u_{2}\|_{B^0_{p,\infty}}\leq C\|u_{1}-u_{2}\|_{L^p}\leq C\|u_{1}(0)-u_{2}(0)\|_{B^{1+\frac{1}{p}}_{p,1}}.	
\end{align*}
So if $u_1(0)=u_2(0)$, we can immediately obtain the uniqueness.

\textbf{Step 3. The continuous dependence}.

Assume that $u^n_{0}$ tends to $u^\infty_{0}$ in $B^{1+\frac{1}{p}}_{p,1}$. Similar to \cite{weikui}, we can find the solutions $u^n,\ u^\infty$ of \eqref{abstract} with a common lifespan $T$. By \textbf{Step 1--Step 2}, we have $u^n,\ u^\infty$ are uniformly bounded in $L^{\infty}\big(0,T;B^{1+\frac{1}{p}}_{p,1}\big)$ and
\begin{align*}
\|\big(u^n-u^\infty\big)(t)\|_{B^0_{p,\infty}}\leq  C\|u_{0}^n-u_{0}^\infty\|_{B^{1+\frac{1}{p}}_{p,1}},\ \forall t\in[0,T].
\end{align*}
Taking advantage of the interpolation inequality, we see that $u^n\rightarrow u^\infty$ in $\mathcal{C}\big([0,T];B^{1+\frac{1}{p}-\epsilon}_{p,1}\big)$ for any $\epsilon>0$. Choosing $\epsilon=1$, we have
\begin{align}
u^n\rightarrow u^\infty\qquad\text{in}\quad \mathcal{C}\big([0,T];B^{\frac{1}{p}}_{p,1}\big).\label{Untend}
\end{align}

In order to prove $u^n\rightarrow u^\infty$ in $\mathcal{C}\big([0,T];B^{\frac{1}{p}+1}_{p,1}\big)$, we next only need to prove $\partial_{x}u^n\rightarrow \partial_{x}u^\infty$ in $\mathcal{C}\big([0,T];B^{\frac{1}{p}}_{p,1}\big)$. For simplicity, let $v^n=\partial_{x}u^n,\ v^\infty=\partial_{x}u^\infty$. Split $v^n$ into $w^n+z^n$ with $(w^n,z^n)$ satisfying
\begin{equation*}
	\left\{\begin{array}{l}
		\partial_{t}w^n+u^n\partial_{x}w^n=\partial_{x}\Big(F(u^{\infty})\Big)-(u^{\infty}_{x})^{2},\\
		w^n(0,x)=v^n_0=\partial_{x}u^n_0
	\end{array}\right.
\end{equation*}
and
\begin{equation*}
\left\{\begin{array}{l}
\partial_{t}z^n+u^{n}\partial_{x}z^n=\partial_{x}\Big(F(u^{n})-F(u^{\infty})\Big)-\big((u^{n}_{x})^2-(u^{\infty}_{x})^2\big)\\
z^n(0,x)=v^n_0-v^\infty_0=\partial_{x}u^n_0-\partial_{x}u^\infty_0.
\end{array}\right.
\end{equation*}
\eqref{fkl} and Lemma \ref{continuous} thus ensure that
\begin{align}
w^n\rightarrow w^\infty\qquad \text{in}\quad  \mathcal{C}\big([0,T];B^{\frac{1}{p}}_{p,1}\big).\label{winfty}
\end{align}
Thanks to \eqref{Fkl}, we have
\begin{align*}
\|\partial_{x}\big(F(u^{n})-F(u^{\infty})\big)\|_{B^{\frac{1}{p}}_{p,1}}
\leq C\|u^n-u^\infty\|_{B^{1+\frac{1}{p}}_{p,1}}
\leq C\Big(\|u^n-u^\infty\|_{B^{\frac{1}{p}}_{p,1}}+\|u^n_x-u^\infty_x\|_{B^{\frac{1}{p}}_{p,1}}\Big)
\end{align*}
and
\begin{align*}
\|(u^{n}_{x}+u^{\infty}_{x})(u^{n}_{x}-u^{\infty}_{x})\|_{B^{\frac{1}{p}}_{p,1}}
\leq  C\|u^n_x-u^\infty_x\|_{B^{\frac{1}{p}}_{p,1}}.
\end{align*}
It follows that for all $n\in\mathbb{N}$,
\begin{align}
\|z^n(t)\|_{B^{\frac{1}{p}}_{p,1}}
\leq&C\Big(\|v^n_0-v^\infty_0\|_{B^{\frac{1}{p}}_{p,1}}+\int_0^t\|u^n-u^\infty\|_{B^{\frac{1}{p}}_{p,1}}+\|u^n_x-u^\infty_x\|_{B^{\frac{1}{p}}_{p,1}}{\ud}\tau\Big)\notag\\
\leq&C\Big(\|v^n_0-v^\infty_0\|_{B^{\frac{1}{p}}_{p,1}}+\int_0^t\|u^n-u^\infty\|_{B^{\frac{1}{p}}_{p,1}}+\|w^n-w^\infty\|_{B^{\frac{1}{p}}_{p,1}}+\|z^n\|_{B^{\frac{1}{p}}_{p,1}}{\ud}\tau\Big).\label{zn}
\end{align}
Using the facts that
\begin{itemize}
\item [-]$v_0^n$ tends to $v_0^{\infty}$ in $B^{\frac{1}{p}}_{p,1}$;
\item [-]$u^n$ tends to $u^{\infty}$ in $\mathcal{C}\big([0,T];B^{\frac{1}{p}}_{p,1}\big)$;
\item [-]$w^n$ tends to $w^{\infty}$ in $\mathcal{C}\big([0,T];B^{\frac{1}{p}}_{p,1}\big)$,
\end{itemize}
and then applying the Gronwall lemma, we conclude that $z^n$ tends to $0$ in $\mathcal{C}\big([0,T];B^{\frac{1}{p}}_{p,1}\big)$. By Lemma \ref{existence}--\ref{priori estimate}, we have
$z^\infty=0$ in $\mathcal{C}\big([0,T];B^{\frac{1}{p}}_{p,1}\big)$.

Therefore,
\begin{align*}
\|v^n-v^\infty\|_{L^\infty\big(0,T;B^{\frac{1}{p}}_{p,1}\big)}\leq&\|w^n-w^\infty\|_{L^\infty\big(0,T;B^{\frac{1}{p}}_{p,1}\big)}+\|z^n-z^\infty\|_{L^\infty\big(0,T;B^{\frac{1}{p}}_{p,1}\big)}\\
\leq&\|w^n-w^\infty\|_{L^\infty\big(0,T;B^{\frac{1}{p}}_{p,1}\big)}+\|z^n\|_{L^\infty\big(0,T;B^{\frac{1}{p}}_{p,1}\big)}\quad\rightarrow 0\quad\text{as}\ n\rightarrow\infty,
\end{align*}
that is
\begin{align*}
\partial_{x}u^n\rightarrow \partial_{x}u^\infty\qquad\text{in}\quad \mathcal{C}\big([0,T];B^{\frac{1}{p}}_{p,1}\big).
\end{align*}
Hence, we prove the continuous dependence of \eqref{abstract} in critial Besov spaces $\mathcal{C}\big([0,T];B^{1+\frac{1}{p}}_{p,1}\big)$ with $p\in[1,+\infty)$.

Consequently, combining with \textbf{Step 1--Step 3}, we finish the proof of Theorem \ref{lagrabstr}.
\end{proof}

\begin{proof}[\rm{\textbf{The proof of Theorem \ref{lagrabstr2}:}}]
The proof of Theorem \ref{lagrabstr2} is similar to that of Theorem \ref{lagrabstr}, and we won't give it here (for the continuous dependence one may find more details in \cite{weikui}).
\end{proof}

\section{Applications}

In this section, we aim to prove Theorem \ref{CH}, Theorem \ref{N} and Theorem \ref{2CH} by applying Theorem \ref{lagrabstr} and Theorem \ref{lagrabstr2}.
\begin{proof}[\rm{\textbf{The proof of Theorem \ref{CH}:}}]
Let $A(u)=u,\ F(u)=-\partial_{x}(1-\partial_{xx})^{-1}\Big(u^2+\frac{1}{2}u^2_x\Big)$. Then \eqref{abstract} becomes the classical CH equation:
\begin{equation}\label{ch}
\left\{\begin{array}{ll}
\partial_{t}u+u\partial_{x}u=-\partial_{x}(1-\partial_{xx})^{-1}\Big(u^2+\frac{1}{2}u^2_x\Big):=-\partial_{x}{\rm{p}}\ast\Big(u^2+\frac{1}{2}u^2_x\Big),&\quad t>0,\quad x\in\mathbb{R},\\
u(t,x)|_{t=0}=u_0(x),&\quad x\in\mathbb{R},
\end{array}\right.
\end{equation}
where ${\rm{p}}(x)=\frac{1}{2}e^{-|x|}$.

According to Theorem \ref{lagrabstr}, we only need to verify \eqref{fkl}--\eqref{Fkl} hold. For $p\in[1,\infty)$, it's not hard to see that $F$ is `good' operator by an approximation argument and
\begin{align*}
&\|F(u)\|_{B^{1+\frac{1}{p}}_{p,1}}\leq C\Big(\|u\|_{B^{1+\frac{1}{p}}_{p,1}}+\|u_{x}\|_{B^{\frac{1}{p}}_{p,1}}\Big)\leq C\Big(\|u\|_{B^{1+\frac{1}{p}}_{p,1}}+1\Big);\\
&\|F(u_1)-F(u_2)\|_{B^{\frac{1}{p}+1}_{p,1}}\leq C\|u_1-u_2\|_{B^{1+\frac{1}{p}}_{p,1}}.	
\end{align*}
So \eqref{fkl} and \eqref{Fkl} hold, which implies \eqref{ch} has a solution $u\in E^p_T$. To prove the local well-posedness, one also need to verify \eqref{wideFkl}.

To prove \eqref{wideFkl}, we establish the Lagrangian coordinates:
\begin{equation}\label{uyxi}
	\left\{\begin{array}{ll}
		\frac{{\ud}}{{\ud}t}y=u\big(t,y(t,\xi)\big),\\
		y(0,\xi)=\xi.
	\end{array}\right.
\end{equation}
Setting $U_1:=u_1(t,y_{1}(t,\xi))$ and $U_2:=u_2(t,y_{2}(t,\xi))$, we see for $i=1,2$,
\begin{align}
	&\frac{{\ud}}{{\ud}t}U_{i}=\widetilde{F}(U_{i},y_{i})=\frac{1}{2}\int_{-\infty}^{+\infty}{\rm sign}\big(y_i(t,\xi)-x\big)e^{-|y_i(t,\xi)-x|}(u_i^2+\frac{1}{2}u_{ix}^2){\ud}x,\label{ut}\\
	&\frac{{\ud}}{{\ud}t}U_{i\xi}=\left(\widetilde{F}(U_{i},y_{i})\right)_{\xi}=U_i^2y_{i\xi}+\frac{1}{2}\frac{U^2_{i\xi}}{y_{i\xi}}-\frac{y_{i\xi}}{2}\int_{-\infty}^{+\infty}e^{-|y_i(t,\xi)-x|}(u_i^2+\frac{1}{2}u_{ix}^2){\ud}x.\label{uxt}
\end{align}
Similar to the proof of Theorem \ref{lagrabstr} in \textbf{Step 2}, for sufficiently small $T>0$, we deduce that
\begin{align}\label{step2}
U_i(t,\xi)\in L^\infty_T(W^{1,p}\cap W^{1,\infty}),~y_i(t,\xi)-\xi\in L^\infty_T(W^{1,p}\cap W^{1,\infty})~ \text{and}~\frac{1}{2}\leq y_{i\xi}\leq C_{u_0}.	
\end{align}

First, we estimate $\|\widetilde{F}(U_{1},y_{1})-\widetilde{F}(U_{2},y_{2})\|_{L^{p}\cap L^{\infty}}$. The main difficulty in it is to estimate
 $$\int_{-\infty}^{+\infty}{\rm sign}\big(y_1(\xi)-x\big)e^{-|y_1(\xi)-x|}u_{1x}^2{\ud}x-\int_{-\infty}^{+\infty}{\rm sign}\big(y_2(\xi)-x\big)e^{-|y_2(\xi)-x|}u_{2x}^2{\ud}x.$$
Since $y_i(i=1,2)$ is monotonically increasing, then ${\rm sign}\big(y_i(\xi)-y_i(\eta)\big)={\rm sign}\big(\xi-\eta\big)$. Thus, we have
\begin{align}
&\int_{-\infty}^{+\infty}{\rm sign}\big(y_1(\xi)-x\big)e^{-|y_1(\xi)-x|}u_{1x}^2{\ud}x-\int_{-\infty}^{+\infty}{\rm sign}\big(y_2(\xi)-x\big)e^{-|y_2(\xi)-x|}u_{2x}^2{\ud}x\notag\\
=&\int_{-\infty}^{+\infty}{\rm sign}\big(y_1(\xi)-y_1(\eta)\big)e^{-|y_1(\xi)-y_1(\eta)|}\frac{U_{1\eta}^2}{y_{1\eta}}{\ud}\eta-\int_{-\infty}^{+\infty}{\rm sign}\big(y_2(\xi)-y_2(\eta)\big)e^{-|y_2(\xi)-y_2(\eta)|}\frac{U_{2\eta}^2}{y_{2\eta}}{\ud}\eta\notag\\
=&\int_{-\infty}^{+\infty}{\rm sign}\big(\xi-\eta\big)e^{-|y_1(\xi)-y_1(\eta)|}\frac{U_{1\eta}^2}{y_{1\eta}}{\ud}\eta-\int_{-\infty}^{+\infty}{\rm sign}\big(\xi-\eta\big)e^{-|y_2(\xi)-y_2(\eta)|}\frac{U_{2\eta}^2}{y_{2\eta}}{\ud}\eta\notag\\
=&\int_{-\infty}^{+\infty}{\rm sign}\big(\xi-\eta\big)\left(e^{-|y_1(\xi)-y_1(\eta)|}-e^{-|y_2(\xi)-y_2(\eta)|}\right)\frac{U_{1\eta}^2}{y_{1\eta}}{\ud}\eta\notag\\
&+\int_{-\infty}^{+\infty}{\rm sign}\big(\xi-\eta\big)e^{-|y_2(\xi)-y_2(\eta)|}\left(\frac{U_{2\eta}^2}{y_{2\eta}}-\frac{U_{1\eta}^2}{y_{1\eta}}\right){\ud}\eta\notag\\
:=&I_1+I_2.\label{u1}
\end{align}
If $\xi>\eta$ $(\text{or}\ \xi<\eta)$, then $y_i(\xi)>y_i(\eta)$ $\big(\text{or}\ y_i(\xi)<y_i(\eta)\big)$. Hence, we gain
\begin{align}
I_1=&\int_{-\infty}^{\xi}\left(e^{-(y_1(\xi)-y_1(\eta))}-e^{-(y_2(\xi)-y_2(\eta))}\right)\frac{U_{1\eta}^2}{y_{1\eta}}{\ud}\eta-\int_{\xi}^{+\infty}\left(e^{y_1(\xi)-y_1(\eta)}-e^{y_2(\xi)-y_2(\eta)}\right)\frac{U_{1\eta}^2}{y_{1\eta}}{\ud}\eta\notag\\
=&\int_{-\infty}^{\xi}e^{-(\xi-\eta)}\left(e^{-\int_0^t(U_1(\xi)-U_1(\eta)){\ud}\tau}-e^{-\int_0^t(U_2(\xi)-U_2(\eta)){\ud}\tau}\right)\frac{U_{1\eta}^2}{y_{1\eta}}{\ud}\eta\notag\\
&-\int_{\xi}^{+\infty}e^{\xi-\eta}\left(e^{\int_0^tU_1(\xi)-U_1(\eta){\ud}\tau}-e^{\int_0^tU_2(\xi)-U_2(\eta){\ud}\tau}\right)\frac{U_{1\eta}^2}{y_{1\eta}}{\ud}\eta\notag\\
\leq&C\|U_1-U_2\|_{L^\infty}\Big[\int_{-\infty}^{\xi}e^{-(\xi-\eta)}\frac{U_{1\eta}^2}{y_{1\eta}}{\ud}\eta+\int_{\xi}^{+\infty}e^{\xi-\eta}\frac{U_{1\eta}^2}{y_{1\eta}}{\ud}\eta\Big]\notag\\
\leq&C\|U_1-U_2\|_{L^\infty}\Big[1_{\geq0}(x)e^{-|x|}\ast\frac{U_{1\eta}^2}{y_{1\eta}}+1_{\leq0}(x)e^{-|x|}\ast\frac{U_{1\eta}^2}{y_{1\eta}}\Big].\label{i1}
\end{align}
In the same way, we have
\begin{align}
I_2\leq C\Big[1_{\geq0}(x)e^{-|x|}\ast\Big(|U_{1\eta}-U_{2\eta}|+|y_{1\eta}-y_{2\eta}|\Big)+1_{\leq0}(x)e^{-|x|}\ast\Big(|U_{1\eta}-U_{2\eta}|+|y_{1\eta}-y_{2\eta}|\Big)\Big].\label{i2}
\end{align}
Combining with the identities \eqref{u1}--\eqref{i2}, we find
\begin{align}
&\int_{-\infty}^{+\infty}{\rm sign}\big(y_1(\xi)-x\big)e^{-|y_1(\xi)-x|}u_{1x}^2{\ud}x-\int_{-\infty}^{+\infty}{\rm sign}\big(y_2(\xi)-x\big)e^{-|y_2(\xi)-x|}u_{2x}^2{\ud}x\notag\\
\leq&C\|U_1-U_2\|_{L^\infty}\Big[1_{\geq0}(x)e^{-|x|}\ast\frac{U_{1\eta}^2}{y_{1\eta}}+1_{\leq0}(x)e^{-|x|}\ast\frac{U_{1\eta}^2}{y_{1\eta}}\Big]\notag\\
&+C\Big[1_{\geq0}(x)e^{-|x|}\ast\Big(|U_{1\eta}-U_{2\eta}|+|y_{1\eta}-y_{2\eta}|\Big)+1_{\leq0}(x)e^{-|x|}\ast\Big(|U_{1\eta}-U_{2\eta}|+|y_{1\eta}-y_{2\eta}|\Big)\Big].\label{i12}
\end{align}
Similar to \eqref{i12}, we get
\begin{align}
	&\int_{-\infty}^{+\infty}{\rm sign}\big(y_1(\xi)-x\big)e^{-|y_1(\xi)-x|}u_{1}^2{\ud}x-\int_{-\infty}^{+\infty}{\rm sign}\big(y_2(\xi)-x\big)e^{-|y_2(\xi)-x|}u_{2}^2{\ud}x\notag\\
	\leq&C\|U_1-U_2\|_{L^\infty}\Big[1_{\geq0}(x)e^{-|x|}\ast\Big( U_{1\eta}^2y_{1\eta}\Big)+1_{\leq0}(x)e^{-|x|}\ast\Big( U_{1\eta}^2y_{1\eta}\Big)\Big]\notag\\
	&+C\Big[1_{\geq0}(x)e^{-|x|}\ast\Big(|U_{1}-U_{2}|+|y_{1\eta}-y_{2\eta}|\Big)+1_{\leq0}(x)e^{-|x|}\ast\Big(|U_{1}-U_{2}|+|y_{1\eta}-y_{2\eta}|\Big)\Big].\label{j12}
\end{align}
It follows from \eqref{ut}, \eqref{i12} and \eqref{j12} that
\begin{align*}
&|\widetilde{F}(U_{1},y_{1})-\widetilde{F}(U_{2},y_{2})|\\
\leq&\Big[1_{\geq0}(x)e^{-|x|}\ast\Big(|U_{1}-U_{2}|+|U_{1\eta}-U_{2\eta}|+|y_{1\eta}-y_{2\eta}|\Big)\Big]\\
&+C\Big[1_{\leq0}(x)e^{-|x|}\ast\Big(|U_{1}-U_{2}|+|U_{1\eta}-U_{2\eta}|+|y_{1\eta}-y_{2\eta}|\Big)\Big]\\
&+C\|U_1-U_2\|_{L^\infty}\Big[1_{\geq0}(x)e^{-|x|}\ast\Big(\frac{U_{1\eta}^2}{y_{1\eta}}+U_{1\eta}^2y_{1\eta}\Big)+1_{\leq0}(x)e^{-|x|}\ast\Big(\frac{U_{1\eta}^2}{y_{1\eta}}+U_{1\eta}^2y_{1\eta}\Big)\Big].
\end{align*}
Combining the Young's inequality and \eqref{step2}, we infer that
\begin{align}
&\|\widetilde{F}(U_{1},y_{1})-\widetilde{F}(U_{2},y_{2})\|_{L^\infty\cap L^p}\notag\\
\leq& C\Big(\|U_1-U_2\|_{L^\infty\cap L^p}+\|U_{1\xi}-U_{2\xi}\|_{L^\infty\cap L^p}+\|y_{1\xi}-y_{2\xi}\|_{L^\infty\cap L^p}\Big).\label{U12}
\end{align}

Then, we estimate $\|(\widetilde{F}(U_{1},y_{1})-\widetilde{F}(U_{2},y_{2}))_{\xi}\|_{L^{p}\cap L^{\infty}}$. Since $-(1-\partial_{xx})^{-1}\partial_{xx}={\rm{Id}}-(1-\partial_{xx})^{-1}$, taking some similar but more simple operations, one can obtain
\begin{align}
	&\|\big(\widetilde{F}(U_{1},y_{1})\big)_{\xi}-\big(\widetilde{F}(U_{2},y_{2})\big)_{\xi}\|_{L^\infty\cap L^p}\notag\\
	\leq& C\Big(\|U_1-U_2\|_{L^\infty\cap L^p}+\|U_{1\xi}-U_{2\xi}\|_{L^\infty\cap L^p}+\|y_{1\xi}-y_{2\xi}\|_{L^\infty\cap L^p}\Big).\label{Ux12}
\end{align}
Thus, we have
\begin{align*}
\|\widetilde{F}(U_{1},y_{1})-\widetilde{F}(U_{2},y_{2})\|_{W^{1,\infty}\cap W^{1,p}}\leq C\Big(\|U_{1}-U_{2}\|_{W^{1,\infty}\cap W^{1,p}}+\|y_{1}-y_{2}\|_{W^{1,\infty}\cap W^{1,p}}\Big),	
\end{align*}
i.e. \eqref{wideFkl} holds. This completes the proof of Theorem \ref{CH}.
\end{proof}

Next, we give another well-posedness result for the CH equation in a more general Banach space which also doesn't contain the peakons  ($B^{1+\frac 1 p}_{p,1}\hookrightarrow B^{1+\frac 1 p}_{p,r}\cap W^{1,\infty},~p,r\in [1,\infty)$):
\begin{theo}\label{general}
	Let $u_0\in B^{1+\frac 1 p}_{p,r}\cap W^{1,\infty}$ with $1\leq p,\ r<\infty$. Then there exists a time $T>0$ such that the CH equation is locally well-posed in $u(t,x)\in \mathcal{C}\big([0,T];B^{1+\frac{1}{p}}_{p,r}\big)\cap \mathcal{C}^1\big([0,T];B^{\frac{1}{p}}_{p,r}\big)$ in the sense of Hadamard.
\end{theo}
\begin{proof}
Let $A(u)=u,\ F(u)=-\partial_{x}(1-\partial_{xx})^{-1}\Big(u^2+\frac{1}{2}u^2_x\Big)$.
Then, we can easily deduce that
\begin{align}
&\|F(u)\|_{B^{1+\frac{1}{p}}_{p,r}\cap W^{1,\infty}}\leq C\|u\|^2_{B^{1+\frac{1}{p}}_{p,r}\cap W^{1,\infty}};\\
&\|\widetilde{F}(U,y)-\widetilde{F}(\bar{U},\bar{y})\|_{W^{1,p}\cap W^{1,\infty}}\leq C_{u_0}\Big(\|U-\bar{U}\|_{W^{1,\infty}\cap W^{1,p}}+\|y-\bar{y}\|_{W^{1,\infty}\cap W^{1,p}}\Big);\\
&\|F(u)-F(\bar{u})\|_{B^{1+\frac{1}{p}}_{p,r}\cap W^{1,\infty}}\leq C_{u_0}\|u-\bar{u}\|_{B^{1+\frac{1}{p}}_{p,r}\cap W^{1,\infty}}.
\end{align}
Similar to the proof of Theorem \ref{lagrabstr}, we can obtain a unique solution $u(t,x)\in \mathcal{C}\big([0,T];B^{1+\frac{1}{p}}_{p,r}\big)\cap \mathcal{C}^1\big([0,T];B^{\frac{1}{p}}_{p,r}\big)$. But only for $W^{1,\infty}$, one can not obtain the continuous dependence since $W^{1,\infty}$ is indivisible, see \cite{lps} for more details.
\end{proof}

Likewise, for the Novikov equation and two-component Camassa-Holm system, one can also obtain the local well-posedness by using Theorem \ref{lagrabstr} and Theorem \ref{lagrabstr2}. It also makes sense for the general Bansch spaces in Theorem \ref{general} .
\begin{proof}[\rm{\textbf{The proof of Theorem \ref{N}:}}]
Let $A(u)=u^{2},\ F(u,u_x)=-(1-\partial_{xx})^{-1}\Big(\partial_{x}\big(\frac{3}{2}uu_{x}^{2}+u^{3}\big)+\frac{1}{2}u_{x}^{3}\Big)$. Then \eqref{abstract} becomes the Novikov equation
\begin{equation*}
\left\{\begin{array}{ll}
u_{t}-u_{xxt}=3uu_{x}u_{xx}+u^{2}u_{xxx}-4u^{2}u_{x},&\quad t>0,\quad x\in\mathbb{R},\\
u(t,x)|_{t=0}=u_0(x),&\quad x\in\mathbb{R}.
\end{array}\right.
\end{equation*}
Applying Theorem \ref{lagrabstr}, we can also gain the well-posedness of the Cauchy problem for Novikov equation in critial Besov space $B^{1+\frac 1 p}_{p,1}$.
\end{proof}

\begin{proof}[\rm{\textbf{The proof of Theorem \ref{2CH}:}}]
Let $A(u)=u,\ F(u,\eta)=-\partial_{x}(1-\partial_{xx})^{-1}(u^{2}+\frac{1}{2}u^{2}_{x}+\frac{1}{2}\eta^{2}+\eta)$ and $h(\eta)=1+\eta$. Then we can get the local well-posedness of the Cauchy problem for two-component Camassa-Holm system
\begin{equation*}
\left\{\begin{array}{ll}
\partial_{t}u+u\partial_{x}u=-\partial_{x}(1-\partial_{xx})^{-1}(u^{2}+\frac{1}{2}u^{2}_{x}+\frac{1}{2}\eta^{2}+\eta),&\quad t>0,\quad x\in\mathbb{R},\\
\partial_{t}\eta+u\partial_{x}\eta=-\partial_{x}u(1+\eta),&\quad t>0,\quad x\in\mathbb{R},\\
u|_{t=0}=u_{0}(x),\ \eta|_{t=0}=\eta_{0},&\quad x\in\mathbb{R}
\end{array}\right.
\end{equation*}
in critial Besov space $B^{1+\frac 1 p}_{p,1}\times B^{\frac 1 p}_{p,1}$ by Theorem \ref{lagrabstr2}.
\end{proof}

\noindent\textbf{Acknowledgements.}
Ye and Yin were partially supported by NNSFC (No. 11671407), FDCT (No. 0091/2013/A3), Guangdong Special Support Program (No. 8-2015)
and the key project of NSF of Guangdong Province (No. 2016A030311004). Guo was partially supported by the Guangdong Basic and Applied Basic Research Foundation (No. 2020A1515111092) and Research Fund of Guangdong-Hong Kong-Macao Joint Laboratory for Intelligent Micro-Nano Optoelectronic Technology (No. 2020B1212030010).

\bibliographystyle{plain}
\bibliography{reference}

\begin{thebibliography}{10}

\bibitem{achm}
M.~S. Alber, H.~Camassa, D.~D. Holm, and J.~E. Marsden.
\newblock The geometry of peaked solitons and billiard solutions of a class of
  integrable {PDE}s.
\newblock {\em Lett. Math. Phys.}, 32(2):137--151, 1994.

\bibitem{book}
H.~Bahouri, J.-Y. Chemin, and R.~Danchin.
\newblock {\em Fourier analysis and nonlinear partial differential equations},
  volume 343 of {\em Grundlehren der Mathematischen Wissenschaften [Fundamental
  Principles of Mathematical Sciences]}.
\newblock Springer, Heidelberg, 2011.

\bibitem{bcz}
A.~Bressan, G.~Chen, and Q.~Zhang.
\newblock Uniqueness of conservative solutions to the {C}amassa-{H}olm equation
  via characteristics.
\newblock {\em Discrete Contin. Dyn. Syst.}, 35(1):25--42, 2015.

\bibitem{bc1}
A.~Bressan and A.~Constantin.
\newblock Global conservative solutions of the {C}amassa-{H}olm equation.
\newblock {\em Arch. Ration. Mech. Anal.}, 183(2):215--239, 2007.

\bibitem{bc2}
A.~Bressan and A.~Constantin.
\newblock Global dissipative solutions of the {C}amassa-{H}olm equation.
\newblock {\em Anal. Appl. (Singap.)}, 5(1):1--27, 2007.

\bibitem{ch}
R.~Camassa and D.~D. Holm.
\newblock An integrable shallow water equation with peaked solitons.
\newblock {\em Phys. Rev. Lett.}, 71(11):1661--1664, 1993.

\bibitem{cht}
C.~S. Cao, D.~D. Holm, and E.~S. Titi.
\newblock Traveling wave solutions for a class of one-dimensional nonlinear
  shallow water wave models.
\newblock {\em J. Dynam. Differential Equations}, 16(1):167--178, 2004.

\bibitem{c1}
A.~Constantin.
\newblock The {H}amiltonian structure of the {C}amassa-{H}olm equation.
\newblock {\em Exposition. Math.}, 15(1):53--85, 1997.

\bibitem{c2}
A.~Constantin.
\newblock Existence of permanent and breaking waves for a shallow water
  equation: a geometric approach.
\newblock {\em Ann. Inst. Fourier (Grenoble)}, 50(2):321--362, 2000.

\bibitem{c3}
A.~Constantin.
\newblock On the scattering problem for the {C}amassa-{H}olm equation.
\newblock {\em R. Soc. Lond. Proc. Ser. A Math. Phys. Eng. Sci.},
  457(2008):953--970, 2001.

\bibitem{c5}
A.~Constantin.
\newblock The trajectories of particles in {S}tokes waves.
\newblock {\em Invent. Math.}, 166(3):523--535, 2006.

\bibitem{ce4}
A.~Constantin and J.~Escher.
\newblock Global existence and blow-up for a shallow water equation.
\newblock {\em Ann. Scuola Norm. Sup. Pisa Cl. Sci. (4)}, 26(2):303--328, 1998.

\bibitem{ce1}
A.~Constantin and J.~Escher.
\newblock Global weak solutions for a shallow water equation.
\newblock {\em Indiana Univ. Math. J.}, 47(4):1527--1545, 1998.

\bibitem{ce5}
A.~Constantin and J.~Escher.
\newblock On the {C}auchy problem for a family of quasilinear hyperbolic
  equations.
\newblock {\em Comm. Partial Differential Equations}, 23(7-8):1449--1458, 1998.

\bibitem{ce3}
A.~Constantin and J.~Escher.
\newblock Wave breaking for nonlinear nonlocal shallow water equations.
\newblock {\em Acta Math.}, 181(2):229--243, 1998.

\bibitem{ce2}
A.~Constantin and J.~Escher.
\newblock Well-posedness, global existence, and blowup phenomena for a periodic
  quasi-linear hyperbolic equation.
\newblock {\em Comm. Pure Appl. Math.}, 51(5):475--504, 1998.

\bibitem{ce6}
A.~Constantin and J.~Escher.
\newblock On the blow-up rate and the blow-up set of breaking waves for a
  shallow water equation.
\newblock {\em Math. Z.}, 233(1):75--91, 2000.

\bibitem{cgi}
A.~Constantin, V.~S. Gerdjikov, and R.~I. Ivanov.
\newblock Inverse scattering transform for the {C}amassa-{H}olm equation.
\newblock {\em Inverse Problems}, 22(6):2197--2207, 2006.

\bibitem{ci}
A.~Constantin and R.~I. Ivanov.
\newblock On an integrable two-component {C}amassa-{H}olm shallow water system.
\newblock {\em Phys. Lett. A}, 372(48):7129--7132, 2008.

\bibitem{cmc}
A.~Constantin and H.~P. McKean.
\newblock A shallow water equation on the circle.
\newblock {\em Comm. Pure Appl. Math.}, 52(8):949--982, 1999.

\bibitem{cmo}
A.~Constantin and L.~Molinet.
\newblock Global weak solutions for a shallow water equation.
\newblock {\em Comm. Math. Phys.}, 211(1):45--61, 2000.

\bibitem{cs}
A.~Constantin and W.~A. Strauss.
\newblock Stability of peakons.
\newblock {\em Comm. Pure Appl. Math.}, 53(5):603--610, 2000.

\bibitem{d1}
R.~Danchin.
\newblock A few remarks on the {C}amassa-{H}olm equation.
\newblock {\em Differential Integral Equations}, 14(8):953--988, 2001.

\bibitem{d2}
R.~Danchin.
\newblock A note on well-posedness for {C}amassa-{H}olm equation.
\newblock {\em J. Differential Equations}, 192(2):429--444, 2003.

\bibitem{em}
K.~El~Dika and L.~Molinet.
\newblock Stability of multipeakons.
\newblock {\em Ann. Inst. H. Poincar\'{e} Anal. Non Lin\'{e}aire},
  26(4):1517--1532, 2009.

\bibitem{ghr1}
K.~Grunert, H.~Holden, and X.~Raynaud.
\newblock Global solutions for the two-component {C}amassa-{H}olm system.
\newblock {\em Comm. Partial Differential Equations}, 37(12):2245--2271, 2012.

\bibitem{gy1}
C.~Guan and Z.~Yin.
\newblock Global existence and blow-up phenomena for an integrable
  two-component {C}amassa-{H}olm shallow water system.
\newblock {\em J. Differential Equations}, 248(8):2003--2014, 2010.

\bibitem{gy2}
C.~Guan and Z.~Yin.
\newblock Global weak solutions for a two-component {C}amassa-{H}olm shallow
  water system.
\newblock {\em J. Funct. Anal.}, 260(4):1132--1154, 2011.

\bibitem{gl2}
G.~Gui and Y.~Liu.
\newblock On the global existence and wave-breaking criteria for the
  two-component {C}amassa-{H}olm system.
\newblock {\em J. Funct. Anal.}, 258(12):4251--4278, 2010.

\bibitem{gl1}
G.~Gui and Y.~Liu.
\newblock On the {C}auchy problem for the two-component {C}amassa-{H}olm
  system.
\newblock {\em Math. Z.}, 268(1-2):45--66, 2011.

\bibitem{glmy}
Z.~Guo, X.~Liu, L.~Molinet, and Z.~Yin.
\newblock Ill-posedness of the {C}amassa-{H}olm and related equations in the
  critical space.
\newblock {\em J. Differential Equations}, 266(2-3):1698--1707, 2019.

\bibitem{hr1}
H.~Holden and X.~Raynaud.
\newblock Global conservative solutions of the {C}amassa-{H}olm equation---a
  {L}agrangian point of view.
\newblock {\em Comm. Partial Differential Equations}, 32(10-12):1511--1549,
  2007.

\bibitem{hr3}
H.~Holden and X.~Raynaud.
\newblock Periodic conservative solutions of the {C}amassa-{H}olm equation.
\newblock {\em Ann. Inst. Fourier (Grenoble)}, 58(3):945--988, 2008.

\bibitem{hw}
A.~N.~W. Hone and J.~Wang.
\newblock Integrable peakon equations with cubic nonlinearity.
\newblock {\em J. Phys. A}, 41(37):372002, 10, 2008.

\bibitem{liy}
J.~Li and Z.~Yin.
\newblock Remarks on the well-posedness of {C}amassa-{H}olm type equations in
  {B}esov spaces.
\newblock {\em J. Differential Equations}, 261(11):6125--6143, 2016.

\bibitem{lio}
Y.~A. Li and P.~J. Olver.
\newblock Well-posedness and blow-up solutions for an integrable nonlinearly
  dispersive model wave equation.
\newblock {\em J. Differential Equations}, 162(1):27--63, 2000.

\bibitem{lps}
F.~Linares, G.~Ponce, and T.~Sideris.
\newblock Properties of solutions to the {C}amassa-{H}olm equation on the line
  in a class containing the peakons.
\newblock In {\em {A}symptotic {A}nalysis for {N}onlinear {D}ispersive and
  {W}ave {E}quations}, pages 197--246. Mathematical Society of Japan, 2019.

\bibitem{n}
V.~Novikov.
\newblock Generalizations of the {C}amassa-{H}olm equation.
\newblock {\em J. Phys. A}, 42(34):342002, 14, 2009.

\bibitem{or}
P.~J. Olver and P.~Rosenau.
\newblock Tri-{H}amiltonian duality between solitons and solitary-wave
  solutions having compact support.
\newblock {\em Phys. Rev. E (3)}, 53(2):1900--1906, 1996.

\bibitem{t}
J.~F. Toland.
\newblock {S}tokes waves.
\newblock {\em Topol. Methods Nonlinear Anal.}, 8(2):413--414 (1997), 1996.

\bibitem{wy2}
X.~Wu and Z.~Yin.
\newblock Well-posedness and global existence for the {N}ovikov equation.
\newblock {\em Ann. Sc. Norm. Super. Pisa Cl. Sci. (5)}, 11(3):707--727, 2012.

\bibitem{wy3}
X.~Wu and Z.~Yin.
\newblock A note on the {C}auchy problem of the {N}ovikov equation.
\newblock {\em Appl. Anal.}, 92(6):1116--1137, 2013.

\bibitem{ylz2}
W.~Yan, Y.~Li, and Y.~Zhang.
\newblock The {C}auchy problem for the {N}ovikov equation.
\newblock {\em NoDEA Nonlinear Differential Equations Appl.}, 20(3):1157--1169,
  2013.

\bibitem{weikui}
W.~Ye, W.~Luo, and Z.~Yin.
\newblock The estimate of lifespan and local well-posedness for the
  non-resistive {MHD} equations in homogeneous {B}esov spaces.
\newblock {\em arXiv preprint arXiv:2012.03489v1}, 2020.

\end{thebibliography}



\end{document}